\documentclass[a4paper,12pt,reqno]{amsart}

\numberwithin{equation}{section}
\usepackage{amsmath,amsthm,amssymb}
\usepackage{mathabx}

\newtheorem{Theorem}{Theorem}
\newtheorem{Definition}[Theorem]{Definition}

\newtheorem{Lemma}{Lemma}[section]
\newtheorem{Proposition}[Theorem]{Proposition}
\newtheorem{Remark}{Remark}[section]
\DeclareMathOperator*{\esssup}{ess\,sup}

\begin{document}
\title
[LWP and Blow-up for half GLK eq in a rough metric]
{Local Well-posedness and Blow-up for the Half Ginzburg-Landau-Kuramoto equation
with rough coefficients and potential}

\author[L. Forcella]{Luigi Forcella}

\address{%
B\^{a}timent des Math\'ematiques,
\'Ecole Polytechnique F\'ed\'eral de Lausanne,
Station 8, CH-1015 Lausanne, Switzerland
}

\email{luigi.forcella@epfl.ch}

\author[K. Fujiwara]{Kazumasa Fujiwara}

\address{%
Centro di Ricerca Matematica Ennio De Giorgi,
Scuola Normale Superiore,
Piazza dei Cavalieri, 3, 56126 Pisa,
Italy
}

\email{kazumasa.fujiwara@sns.it}

\thanks{
K. Fujiwara was partly supported
by Top Global University Project of Waseda University.
}

\author[V. Georgiev]{Vladimir Georgiev}

\address{%
Department of Mathematics,
University of Pisa,
Largo Bruno Pontecorvo 5
I - 56127 Pisa, Italy,
Faculty of Science and Engineering, Waseda University,
3-4-1, Okubo, Shinjuku-ku, Tokyo 169-8555,
Japan,
and IMI--BAS, Acad.
Georgi Bonchev Str., Block 8, 1113 Sofia, Bulgaria
}

\email{georgiev@dm.unipi.it}

\thanks{V. Georgiev was supported in part by  INDAM,
GNAMPA - Gruppo Nazionale per l'Analisi Matematica,
la Probabilit\`a e le loro Applicazioni,
by Institute of Mathematics and Informatics,
Bulgarian Academy of Sciences and Top Global University Project, Waseda University.}

\author[T. Ozawa]{Tohru Ozawa}

\address{%
Department of Applied Physics, Waseda University,
3-4-1, Okubo, Shinjuku-ku, Tokyo 169-8555, Japan}

\email{txozawa@waseda.jp}
\thanks{T. Ozawa was supported by
Grant-in-Aid for Scientific Research (A) Number 26247014.}

\begin{abstract}
We study the Cauchy problem for the half Ginzburg-Landau-Kuramoto (hGLK) equation
with the second order elliptic operator having rough coefficients
and potential type perturbation.
The blow-up of solutions for hGLK equation
with non-positive nonlinearity is shown by an ODE argument.
The key tools in the proof are appropriate commutator estimates
and the essential self-adjointness
of the symmetric uniformly elliptic operator with rough metric and potential type perturbation.
\end{abstract}

\date{}

\keywords{fractional Ginzburg-Landau equation, commutator estimate, blow-up}
\subjclass[2010]{Primary 35Q40; Secondary 35Q55}

\maketitle

\tableofcontents

\section{Introduction}
In this paper,
we study the Cauchy problem for the focusing half Ginzburg-Landau-Kuramoto (hGLK)
type equation
	\begin{align}
	i \partial_t u + \mathcal D_{A,V} u =  i |u|^{p-1} u,\qquad p>1.
	\label{eq:1.1}
	\end{align}
Here $\mathcal D_{A,V}$ is the fractional Hamiltonian
(see \cite{La00} for a more general choice of the fractional powers
of the Laplacian)
	\[
	\mathcal D_{A,V} = \mathcal H_{A,V}^{1/2},
	\]
where
	\[
	\mathcal H_{A,V} \mid_{C_c^\infty(\mathbb R^n)}
	= - \Delta_{A,V}
	= - \nabla \cdot A \nabla + V
	= -\sum_{j,k=1}^n \partial_j (A_{j,k}(x) \partial_k) + V
	\]
is a self-adjoint non-negative operator with a real-valued potential,
such that the positive Hermite matrix $A$ and the potential $V$
satisfy appropriate assumptions given below.
The fractional power of $\mathcal H_{A,V}$ is defined by spectral analysis.
For details, see Definition \ref{Definition:6} below.
Beside the other ones, it is worth mentioning
that $A$ is supposed to ensure that $\mathcal H_{A,0}$ is an elliptic second order operator
in divergence form.
Furthermore, \emph{focusing} stands for the ``$+$" sign
in front of the nonlinearity in \eqref{eq:1.1}.

We recall that the classical Ginzburg-Landau equation is
instead typically associated with the standard Laplacian as Hamiltonian
(see \cite{vS92} for a recent review and references on this classical subject).

The idea to replace the Laplace operator in the Hamiltonian
of some quantum mechanical models by its fractional powers was initiated in \cite{La00}
and has been intensively studied in the last decade
(see \cite{TZ16}, for instance,
for motivations to take the square root of the Laplacian
and for an overview of the results in this context).

The half Ginzburg-Landau-Kuramoto equation \eqref{eq:1.1},
which is the main subject of this paper,
is closely connected with the Kuramoto model (see \cite{Kur00}, \cite{ACGM16})
and the idea (proposed in \cite{La00} and \cite{TZ16})
to use the square root of the Laplacian in the definition of the Hamiltonian.

In order to define $\mathcal D_{A,V}$,
we need to prove that $-\Delta_{A,V}$ has a self-adjoint extension,
where we regard the domain of $-\Delta_{A,V}$ as $C_c^\infty(\mathbb R^n)$.
One can find a self-adjoint extension for $-\Delta_{A,V}$
with rough coefficients $A$ and rough potential $V$
by using the Friedrichs type extension under the non-negativity assumption
(see \cite[Theorem 1.2.7]{D90}).
Recall that the domain of Friedrichs type extension can be defined
as the set of all $f \in H^1(\mathbb{R}^n)$,
such that there exists $g \in L^2(\mathbb R^n) $ satisfying
	\begin{equation}\label{eq:1.2}
	- \Delta_{A,V} f = g
	\end{equation}
in distributional sense.
On the other hand,
since the argument of Friedrichs type extension does not guarantee
the uniqueness of self-adjoint extensions,
in order to clarify the definition of fractional power of $\mathcal H_{A,V}$,
we also need to show the uniqueness of self-adjoint extensions of $-\Delta_{A,V}$.
In this case,
we say that the operator $- \Delta_{A,V}$ is essentially self-adjoint
(the problem is referred to as \emph{quantum completeness}, too).
Some sufficient conditions for the essential self-adjointness
for general symmetric operators on manifolds have been discussed in \cite{BMS02},
for instance.
In this paper,
we give a detailed proof of the essential self-adjointness of $-\Delta_{A,V}$
(see the Subsection \ref{section:1.2} below for the precise hypothesis).

We started the study of this model in \cite{FGO17b},
where local and global well-posedness were discussed for the \emph{defocusing}
(``$-$" sign in front of the nonlinearity) equation
	\[
	i \partial_t u + (-\Delta)^{1/2} u =  -i |u|^{p-1} u
	\]
in space dimensions $n=1,2,3.$ The blow-up result for the focusing equation
	\[
	i \partial_t u + (-\Delta)^{1/2} u =  i |u|^{p-1} u
	\]
is obtained instead in \cite{FGO17c} for $n=1$.
In \cite{FGO17c},
the proof of the blow-up result uses the following simple commutator estimates:
	\[
	\| [(-\Delta)^{1/2},f] g \|_{L^2}\leq C \| f \|_{\mathrm{Lip}} \| g\|_{L^2},
	\]
where $f$ is a Lipschitz function with corresponding norm $\| f \|_{\mathrm{Lip}}$.
In order to show the blow-up of solutions to \eqref{eq:1.1},
we shall prove the following estimates
	\begin{align}
	\| [ f , \mathcal D_{A,0}] g \|_{L^2}
	\leq C \| f \|_{\dot B_{\infty,1}^1}
	\| g \|_{L^2},
	\label{eq:1.3}\\
	\| [ f , \mathcal D_{A,V}] g \|_{L^2}
	\leq C \| f \|_{B_{\infty,1}^1}
	\| g \|_{L^2},
	\label{eq:1.4}
	\end{align}
where $B^s_{p,q}$ and $\dot{B}^s_{p,q}$ are the standard
inhomogeneous and homogeneous Besov spaces on $\mathbb{R}^n$,
respectively.
Since
	\[
	B_{\infty,1}^1 \cup \dot{B}_{\infty,1}^1 \subsetneq \mathrm{Lip},
	\]
it would be natural to pose the question
if the estimates \eqref{eq:1.3} and \eqref{eq:1.4} are optimal
for the case of rough coefficients;
but this is not our goal,
hence we do not investigate this question, as well as the question if the commutator
	\[
	[ \mathcal D_{A,V}, \langle x \rangle ]
	\]
is a bounded operator in $L^2$.
However, by replacing $\langle x \rangle$ by $ \langle x \rangle^a$,
our aim shall be to check that the commutator
	\[
	[ \mathcal D_{A,V}, \langle x \rangle^a ]
	\]
is an $L^2$-bounded operator for any $a \in (1/2,1)$
and this shall be a sufficient tool to obtain our blow-up result at least for $n=1$.

\subsection{Notations}
We collect here some notations used along the paper.
Given two quantities $A$ and $B$,
we denote $A \lesssim B$ ($A\gtrsim B$, respectively)
if there exists a positive constant $C$ such that $A \leq C B$ ($A\geq C B$, respectively).
We also denote $A \sim B$ if $A \lesssim B \lesssim A$.
Given two operators $\mathcal M$ and $\mathcal N$,
the commutator between them is defined as the operator
$[\mathcal M,\mathcal N]=\mathcal M \mathcal N - \mathcal N \mathcal M$.
For $1\leq p\leq \infty$,
the $L^p=L^p(\mathbb R^n;\mathbb C)$ are the classical Lebesgue spaces
endowed with norm $\|f\|_{L^p}=\left(\int_{\mathbb R^n}|f(x)|^p\,dx\right)^{1/p}$
if $p\neq\infty$ or $\|f\|_{L^\infty}=\esssup_{x\in\mathbb R^n}|f(x)|$ for $p=\infty$.
Given an interval $I \subset \mathbb R$,
bounded or unbounded,
we define by $L^p(I;X)$ the Bochner space of vector-valued functions $f:I\to X$
endowed with the norm $\left(\int_{I}\|f(s)\|_{X}^p\,dx\right)^{1/p}$
for $1\leq p<\infty$,
with similar modification as above for $p=\infty$.
If $f:I\to X$ is a continuous function up to the $m^{\mathrm{th}}$-order of derivatives,
we write $f\in C^m(I;X)$.
For any $s\in\mathbb R$,
we set $H^s=H^s(\mathbb R^n;\mathbb C):=(1-\Delta)^{-s/2}L^2$
and its homogeneous version
$\dot H^s=\dot H^s(\mathbb R^n;\mathbb C):=(-\Delta)^{-s/2}L^2$.
For a pair of functions in $L^2$,
the inner product $\langle f,g\rangle=\langle f,g\rangle_{L^2}$ is classically
defined as $\langle f,g\rangle=\int_{\mathbb R^n}f\bar g\,dx,$
being $\bar z$, the usual complex conjugate to $z \in \mathbb C$.
For $x\in\mathbb R^n$ instead, $\langle x\rangle :=\sqrt{1+|x|^2}$.
The space $W^{1,\infty}=W^{1,\infty}(\mathbb R^n)$ is the space of Lipschitz functions.
The operator $\mathfrak F f(\xi)=\hat f(\xi)$ is the standard Fourier transform,
$\mathfrak F^{-1}$ being its inverse.
For $s\in\mathbb R$ and $0<p,q\leq\infty$,
$\dot{B}^s_{p,q}=\dot{B}^s_{p,q}(\mathbb R^n)$ is the homogeneous Besov space
of functions having finite $\|\cdot\|_{\dot{B}^s_{p,q}}$-norm,
the last defined as
	\[
	\|f\|_{\dot{B}^s_{p,q}}=\left(\sum_{j\in\mathbb Z}2^{sjq}\|P_jf\|_{L^p}^q\right)^{1/q}
	\]
with obvious modifications for $p,q=\infty$.
The non-homogeneous version $B^s_{p,q}=B^s_{p,q}(\mathbb R^n)$ is induced
by the norm
	\[
	\|f\|_{B^s_{p,q}}
	=\|Q f\|_{L^p}+\left(\sum_{j\in\mathbb N}2^{sjq}\|P_jf\|_{L^p}^q\right)^{1/q}.
	\]
Here the Littlewood-Paley projectors $P_j$ are defined
by means of a radial cut-off function $\chi_0\in C^{\infty}_c(\mathbb R^n)$ and the dyadic functions
$ \varphi_j(\xi)=\chi_0(2^{-j}\xi)-\chi_0(2^{-j+1}\xi)$
yielding to the partition of the unity $\chi_0(\xi)+\sum_{j \geq 1} \varphi_j(\xi)=1$,
for any $\xi\in\mathbb R^n$.
Hence the projectors are given by
$Q f := \mathfrak{F}^{-1}\left(\chi_0 \mathfrak F f\right)$
and $P_j f := \mathfrak{F}^{-1}\left(\varphi_j \mathcal{F}f\right)$.
The Lorentz space $L^{\beta,\infty}$ is given by
	\[
	L^{\beta,\infty}
	=\{ f \, : \, \| f \|_{L^{\beta,\infty}}^\beta = \sup_{t > 0} t^\beta |\{|f|>t\}| < \infty \}.
	\]
For $1\leq p\leq \infty$, $p^\prime$ is the conjugate index defined by
$1/p + 1/p' = 1$.

\subsection{Assumptions and the main results}
\label{section:1.2}
We give now the precise assumptions that we make on the structure
of our Hamiltonian $-\Delta_{A,V}$ and the main results contained in the paper.
We start with the hypotheses on $A=A(x)$, which is a Hermitian matrix-valued function. We assume:
\begin{itemize}
\item[A1.]
Uniform ellipticity of $A$:
There exist two positive constants $C_1$ and $C_2$ satisfying
	\begin{align}
	&C_1 |\xi|^2
	\leq \sum_{j,k=1}^n
	A_{j,k}(x) \xi_j \overline \xi_k
	\leq C_2 |\xi|^2,\qquad \forall\,\xi\in\mathbb C^n,\quad \forall\,x\in\mathbb R^n;
	\label{eq:1.5}
	\end{align}
\item[A2.]
Regularity of the coefficients:
$A$ is in the Lipschitz class of matrix-valued functions, namely
	\[
	A_{j,k}\in W^{1,\infty}(\mathbb R^n)\qquad j,k\in\{1,\cdots, n\};
	\]
\item[A3.]
Boundedness:
The multiplication operator
	\[
	f \mapsto \ ((-\Delta)^{1/4}A_{j,k} )f
	\]
maps $\dot{H}^{1/2}$ into $L^2$, namely
	\begin{equation}
	\max_{j,k} \|
	((-\Delta)^{1/4}A_{j,k}) f \|_{L^2}
	\leq C \| (-\Delta)^{1/4} f \|_{L^2},\qquad \forall\, f\in \dot H^{1/2}.
	\label{eq:1.6}
	\end{equation}
\end{itemize}

Let us turn our attention to the potential perturbation $V=V(x)$.
It is a real-valued function satisfying the following conditions:

\begin{itemize}
\item[H1.] Boundedness of the potential:
	\[
	V \in L^{q,\infty}(\mathbb{R}^n) + L^\infty(\mathbb{R}^n)
	\]
for some $q$ with $q > \max\{2,n/2\}$;

\item[H2.] Non-negativity of the Hamiltonian $- \Delta_{ A,V}$:
There exists $\theta \in (0,1)$ such that
	\[
	\theta \langle  A \nabla f, \nabla f \rangle + \langle V f, f \rangle \geq 0,
	\qquad \forall\,f \in C_0^\infty(\mathbb{R}^n).
	\]
\end{itemize}

Though for the moment it is not our aim to weaken the non-negativity assumption in H2,
it is worth mentioning that this hypothesis could be relaxed, at least in the case $n=1$.
For example for $A=1$, perturbations of the Laplacian which belongs to the Miura
class should imply positivity of such Hamiltonians (see \cite{KPST05}).
The non-negativity assumption is needed to guarantee
that the square root of the operator is well defined.

First we state the result on the self-adjoint extension of the operator $-\Delta_{A,V}$.
This theorem is crucial for the local well-posedness theory below
and for the commutator estimates we are going to prove.
\begin{Theorem}
\label{Theorem:1}
Assume the  assumptions $\mathrm{A1, A2, H1}$ and $\mathrm{H2}$ are satisfied.
Then the operator $-\Delta_{A,V}$ is essentially self-adjoint,
i.e. there exists a unique self-adjoint extension $ \mathcal{H}_{A,V}$
of this operator with domain
	\[
	D( \mathcal{H}_{A,V}) = H^2(\mathbb{R}^n).
	\]
\end{Theorem}

\medskip

The key point in our blow-up result shall be instead the following commutator estimate.
\begin{Proposition}
\label{Proposition:2}
Assume the  conditions $\mathrm{A1, A2, A3}$ and $\mathrm{H1, H2}$ are satisfied.
We have the two commutator estimates in two cases below.\\

\noindent
\emph{Case 1.} Let $f \in B_{\infty,1}^1$.
If  $V$ also belongs to $L^{q, \infty}$ for $q > \max\{2,n\}$,
	\begin{align}
	\| [ f , \mathcal D_{A,V}] g \|_{L^2}
	\leq C \left( \| f \|_{\dot B_{\infty,1}^1}
	+ \| V \|_{L^{q,\infty}} \|
	f \|_{L^\infty} \right) \| g \|_{L^2}.
	\label{eq:1.7}
	\end{align}

\noindent
\emph{Case 2.} Suppose $n\geq3$ and  $f \in \dot B_{\infty,1}^1$.
If $V$ also belongs to $L^{n/2,\infty}$ then,
	\begin{align}
	\| [ f , \mathcal D_{A,V}] g \|_{L^2(\mathbb R^n)}
	\leq C \| f \|_{\dot B_{\infty,1}^1(\mathbb R^n)}
	\| g \|_{L^2(\mathbb R^n)}.
	\label{eq:1.8}
	\end{align}
\end{Proposition}

\medskip

Next we turn to the local well-posedness of \eqref{eq:1.1}.
\begin{Theorem}
\label{Theorem:3}
Let $n=1,2,3$.
Assume that the  conditions $\mathrm{A1, A2, H1}$ and $\mathrm{H2}$ are satisfied.
Then for any $u_0 \in H^s$ with $s=1$ if $n = 1$ or $s = 2$ if $n = 2,3$,
there exists a positive time $T>0$
and a solution $u \in C([0,T);H^s)$ to \eqref{eq:1.1}.
\end{Theorem}

\medskip

The next result is the finite time blow-up result for solutions to \eqref{eq:1.1}
in one space dimension.
\begin{Theorem}
\label{Theorem:4}
Let $n=1$.
Assume the conditions $\mathrm{A1, A2, A3}$ and $\mathrm{H2}$
are satisfied and  $V \in L^{q, \infty}$ for some $q$ with $q > 2$.\\

\noindent
\emph{Case 1.} Let $u_0 \in L^2$ and $w \in B_{\infty,1}^1$
satisfy $1/w \in L^\infty \cap L^2$ and the following estimate:
	\begin{align}
	\| w u_0 \|_{L^2}^2
	\geq C^{\frac{2}{p-1}} \| 1/w
	\|_{L^\infty}^{\frac{2}{p-1}}
	\| w \|_{B_{\infty,1}^1}^{\frac{2}{p-1}}
	\| 1/w \|_{L^{2}}^{2}.
	\label{eq:1.9}
	\end{align}
If there exists a
solution $u \in C(\lbrack0,T_{\mathrm{max}});L^2\cap L^{p+1})$,
then the maximal time of existence is finite: $T_{\mathrm{max}}<\infty$.\\

\noindent
\emph{Case 2.} Suppose that $V\equiv0.$ Let $1 < p < 3$ and let
$u_0 \in L^2 \backslash \{0\}.$ If there exists a solution $u \in C(\lbrack0,T_{\mathrm{max}});L^2\cap L^{p+1})$, then the maximal  time of existence is finite:
$T_{\mathrm{max}} < \infty$.
\end{Theorem}

\medskip

\begin{Remark}
Here the condition $p=3$ corresponds to the critical exponent $p_F=1+2/n$
defined also in a multidimensional framework.
See the results in \cite{FGO17c}.
\end{Remark}

\section{Self-Adjointness of $-\Delta_{A,V}$}
The proof of Theorem \ref{Theorem:1} can be reduced to the proof
that $-\Delta_{A,0}$ is essentially self-adjoint.
Indeed, if $-\Delta_{A,0}$ has unique self-adjoint extension $\mathcal{H}_{A,0}$
with domain
	\[
	D(\mathcal{H}_{A,0}) = H^2(\mathbb{R}^n),
	\]
then we can use the estimate \eqref{eq:A.1} of Lemma \ref{Lemma:A.1}
in combination with the KLMN lemma (see \cite[Theorem X.17]{RSII})
and deduce that $\mathcal{H}_{A}+V$ is an essentially self-adjoint operator
with domain $H^2(\mathbb{R}^n).$

Therefore,
it remains to verify that $-\Delta_{A,0}$ is essentially self-adjoint.
This is done below in Proposition \ref{Proposition:5}
and this yields to Theorem \ref{Theorem:1}.
Firstly, we recall sufficient equivalent conditions
guaranteeing the self-adjointness property of an operator.
\begin{Lemma}{\cite[Theorem X.26]{RSII}}\label{Lemma:2.1}
Assume the operator $-\Delta_{A,V}$ is non-negative
(in sense of quadratic form acting on $C_0^\infty$ functions).
Then the following conditions are equivalent:
\begin{itemize}
\item[i)]
$-\Delta_{A,V}$  is essentially self-adjoint;
\item[ii)]
the kernel of the adjoint operator satisfies
	\[
	\mathrm{Ker}\,(-\Delta_{A,V} +1)^* = \{0\};
	\]
\item[iii)]
the range of $ (-\Delta_{A,V} +1)$ is dense in $L^2(\mathbb{R}^n)$:
	\begin{equation}\label{eq:2.1}
	\overline{\lbrack \mathrm{Ran} \, (-\Delta_{A,V} +1) \rbrack}
	= L^2(\mathbb{R}^n).
	\end{equation}
\end{itemize}
\end{Lemma}

\medskip

Next, we recall some fundamental operator calculus.
\begin{Lemma}\label{Lemma:2.2}
For $f$ smooth enough,
	\begin{align}
	[-\Delta_{A,V} ,f]
	= (\nabla f) \cdot A \nabla + \nabla \cdot A (\nabla f).
	\label{eq:2.2}
	\end{align}
\end{Lemma}

\medskip

\begin{proof}
For completeness, we shall sketch the proof.
The relation \eqref{eq:2.2}  follows directly from the simple commutator rule
	\begin{align*}
	[B_1 B_2, f ]
	&= B_1B_2 f - f B_1 B_2 -B_1 f B_2 + B_1 f B_2\\
	&= B_1[B_2, f]+ [B_1,f]B_2.
	\end{align*}
\end{proof}

\begin{Lemma}\label{Lemma:2.3}
Let $\mathcal A$ be non-negative self-adjoint operator.
Then
	\[
	[(\lambda + \mathcal A)^{-1},f]
	=(\lambda + \mathcal A)^{-1} [\mathcal A, f] (\lambda + \mathcal A)^{-1}.
	\]
\end{Lemma}

\begin{proof}
For completeness, we shall sketch the proof.
Noting the identity
	\begin{align*}
	0
	&= [(\lambda + \mathcal A)(\lambda + \mathcal A)^{-1},f]\\
	&= (\lambda + \mathcal A)[(\lambda + \mathcal A)^{-1}, f]
	+ [\mathcal A ,f](\lambda + \Delta_{A,V})^{-1}
	\end{align*}
and applying the resolvent $(\lambda + \mathcal A)^{-1}$ from the left,
we obtain the assertion.
\end{proof}

We can now give the following:

\begin{Proposition}
\label{Proposition:5}
Assume the assumptions $\mathrm{A1}$ and $\mathrm{A2}$ are satisfied.
Then the operator $-\Delta_{A,0}$ is essentially self-adjoint,
i.e. there exists a unique self-adjoint extension $ \mathcal{H}_{A}$
of this operator with domain
	\[
	D( \mathcal{H}_{A}) = H^2(\mathbb{R}^n).
	\]
\end{Proposition}

\medskip

\begin{proof}
We show that the closure $\overline{(-\Delta_{A,0})}$ is self-adjoint.
Lemma \ref{Lemma:A.2} in the Appendix \ref{section:A} below, implies that
	\[
	D(\overline{(-\Delta_{A,0})}) = H^2(\mathbb{R}^n).
	\]

Thanks to symmetry and regularity of $A$,
there exists at least one self-adjoint extension of $-\Delta_{A,0}$.
Indeed,
since
$- \Delta_{A,0}$ is symmetric and $A \in W^{1,2}_{loc}(\mathbb{R}^n)$,
the quadratic form
	\[
	Q(f)
	= \sum_{j, k=1}^n \int_{\mathbb{R}^n} A_{j k}(x)
	\partial_{x_j} f(x) \overline{\partial_{x_k} f(x)}\,dx,
	\quad
	D(Q) = H^1(\mathbb R^n)
	\]
is closable and possesses a self-adjoint operator $\mathcal{H}_A$ satisfying 
	\begin{align*}
	Q(f)
	&= \langle \mathcal H_A f, f \rangle
	\end{align*}
for any $f \in D(\mathcal H_A) \subset H^1(\mathbb R^n)$
(see \cite[Theorem 1.2.5]{D90}).
We recall that any self-adjoint extension of $-\Delta_{A,0}$
is also an extension of $\overline{(-\Delta_{A,0})}$ and so is $\mathcal H_A$.

Now we show that
the self-adjointness of $\mathcal H_A$ implies that also $\overline{(-\Delta_{A,0})}$ is self-adjoint.
For this purpose,
we shall check the equivalent assertion \eqref{eq:2.1}
in Lemma \ref{Lemma:2.1}.
Let $h \in L^2$ satisfy
	\begin{equation}\label{eq:2.3}
	h \perp \textrm{Ran} ( - \Delta_{A,0} + 1),
	\end{equation}
namely $h$ is orthogonal to the range of $( - \Delta_{A,0} + 1)$.
Our goal is to show that $h=0$.
We define the Yosida type approximation of Laplacian
	\[
	\rho_j = j (j- \Delta)^{-1}
	\]
with $j \geq 1$.
We show that
	\[
	\overline{(-\Delta_{A,0})} ( \rho_j f)
	\stackrel{L^2}{\rightharpoonup} \mathcal{H}_A f,
	\quad
	\forall f \in D(\mathcal H_A).
	\]
We remark that, for any $j \geq 1$ and $f \in L^2$, $\rho_j f \in H^2$.
For any $b \in L^2$,
	\begin{align*}
	\langle \overline{(-\Delta_{A,0})} \rho_j f, b \rangle
	&= \lim_{k \to \infty} \langle \overline{(-\Delta_{A,0})} \rho_j f, \rho_k b \rangle\\
	&= \lim_{k \to \infty} \langle f, \rho_j \overline{(-\Delta_{A,0})} \rho_k  b \rangle\\
	&= \lim_{k \to \infty}
	( \langle f, \overline{(-\Delta_{A,0})} \rho_j \rho_k b \rangle + R_{j,k} ),
	\end{align*}
where
$R_{j,k} = [\rho_j, \overline{(- \Delta_{A,0})}] \rho_k$.
Since for $g \in H^2(\mathbb R^n)$,
$\overline{(- \Delta_{A,0})} g = \nabla \cdot A \nabla g$
in the distributional sense and $\rho_j$ commutes with $\nabla$,
by Lemma \ref{Lemma:2.2},
	\begin{align*}
	[\rho_j, \overline{(- \Delta_{A,0})}]
	&= j \nabla \cdot [(j - \Delta)^{-1}, A ] \nabla\\
	&= j (j - \Delta)^{-1} \nabla \cdot [ - \Delta, A ] \nabla (j - \Delta)^{-1},
	\end{align*}
where
	\[
	[\Delta,A]_{j,k}
	= (\nabla A_{j,k}) \cdot \nabla
	+ \nabla \cdot (\nabla A_{j,k}).
	\]
Therefore
	\begin{align*}
	| R_{j,k} |
	&= j |\langle \nabla f,
	(j-\Delta)^{-1} [A, \Delta] \nabla (j-\Delta)^{-1} \rho_k b \rangle|\\
	&= \sum_{m_1,m_2,m_3} j |\langle \partial_{m_1} f,
	(j-\Delta)^{-1} (\partial_{m_2} A_{m_1,m_3})
	\rho_k \partial_{m_2} \partial_{m_3} (j-\Delta)^{-1} b \rangle|\\
	&+ \sum_{m_1,m_2,m_3} j |\langle \partial_{m_1} f,
	(j-\Delta)^{-1} \partial_{m_2} (\partial_{m_2} A_{m_1,m_3})
	\rho_k \partial_{m_3} (j-\Delta)^{-1} b \rangle|\\
	&\leq n^3
	\| \nabla A\|_{L^\infty}
	\| \rho_j \nabla f \|_{L^2}
	\| \rho_k \nabla \otimes \nabla (j-\Delta)^{-1} b \|_{L^2}\\
	&+ n^3
	\| \nabla A\|_{L^\infty}
	\| j^{1/2}  (j-\Delta)^{-1} \nabla \otimes \nabla f \|_{L^2}
	\|\rho_k j^{1/2}  (j-\Delta)^{-1} \nabla b \|_{L^2}.
	\end{align*}
Recall that for $g \in L^2$
	\[
	j^{1/2} \nabla (j-\Delta)^{-1} g \to 0
	\qquad \mathrm{in} \quad L^2
	\]
and
	\[
	\nabla \otimes \nabla (j-\Delta)^{-1} g \to 0
	\qquad \mathrm{in} \quad L^2,
	\]
where
	\[
	( \nabla \otimes \nabla )_{m,\ell}
	= \partial_m \partial_\ell.
	\]
Since $f \in H^1$, we get
	\begin{align*}
	&n^{-3} \limsup_{j \to \infty} \limsup_{k \to \infty}| R_{j,k} |\\
	&\leq \limsup_{j \to \infty}
	\| \nabla A\|_{L^\infty}
	\| \rho_j \nabla f \|_{L^2}
	\| \nabla \otimes \nabla (j-\Delta)^{-1} b \|_{L^2}\\
	&+ \limsup_{j \to \infty}
	\| \nabla A\|_{L^\infty}
	\| j^{1/2} (j-\Delta)^{-1} \nabla \otimes \nabla f \|_{L^2}
	\| j^{1/2} (j-\Delta)^{-1} \nabla b \|_{L^2}\\
	&\leq \limsup_{j \to \infty}
	\| \nabla A\|_{L^\infty}
	\| \nabla f \|_{L^2}
	\| \nabla \otimes \nabla (j-\Delta)^{-1} b \|_{L^2}\\
	&+ \limsup_{j \to \infty}
	\| \nabla A\|_{L^\infty}
	\| \nabla f \|_{L^2}
	\| j^{1/2} (j-\Delta)^{-1} \nabla b \|_{L^2}=0.
	\end{align*}
Moreover, as $j\to\infty,$
	\[
	\lim_{k \to \infty} \langle f, \overline{(-\Delta_{A,0})} \rho_j \rho_k b \rangle
	= \lim_{k \to \infty} \langle \rho_j \mathcal H_A f, \rho_k b \rangle
	= \langle \rho_j \mathcal H_A f, b \rangle
	\to \langle \mathcal H_A f, b \rangle.
	\]
Hence,  if $h$ satisfies \eqref{eq:2.3}, then for any $f \in D(\mathcal H_A)$,
	\[
	\langle ( \mathcal H_A +1 ) f, h \rangle
	= \lim_{j \to \infty}\langle ( ( - \Delta_{A,0}) +1 ) \rho_j f, h \rangle
	= 0.
	\]
Therefore,
$ h \perp \mathrm{Ran}( \mathcal H_A +1  )$
and the self-adjointness of $\mathcal H_A$ implies $h = 0$.
\end{proof}

\section{Local well-posedness of \eqref{eq:1.1}}\label{section:3}
This section is devoted to the proof of the local well-posedness
for the Cauchy problem associated with the model \eqref{eq:1.1},
where $u_0(x)=u(0,x)$ is considered as initial datum.
More precisely, we give now a proof of Theorem \ref{Theorem:3}.

At first,
we give the definition of $\mathcal D_{A,V}$.
We use a functional calculus for the fractional powers of self-adjoint operators
based on the integral representation below (see (4.7) in \cite{GMS17}, for example).
\begin{Definition}
\label{Definition:6}
Let $\mathcal A$ be a non-negative self-adjoint operator.
For $0 < s < 2$,
	\begin{equation}\label{eq:3.1}
	\mathcal A^{s/2}= C_0(s)
	\int_0^\infty \lambda^{s/2-1} \mathcal A
	(\lambda + \mathcal A )^{-1} d\lambda
	\end{equation}
where
	\[
	C_0(s)= \bigg( \int_0^\infty \lambda^{s/2-1}
	(\lambda + 1)^{-1} d\lambda \bigg)^{-1} = \frac{\sin \left(s\frac{\pi}{2}\right)}{\pi}.
	\]
\end{Definition}

\medskip

We remark that here the formula
	\begin{equation}\label{eq:3.2}
	x^{s/2}\;=\;\frac{\sin\left(s\frac{\pi}{2}\right)}{\pi}
	\int_0^{+\infty} t^{s/2-1}\,\frac{x}{t+x}\, dt ,\qquad x\geqslant 0\,,\quad s\in(0,2)\,.
	\end{equation}
plays a critical role.

Now we can conclude this section by proving Theorem \ref{Theorem:3}.
\begin{proof}[Proof of Theorem \ref{Theorem:3}]
We rewrite \eqref{eq:1.1} in the integral form by means of its Duhamel's formulation
	\begin{equation}\label{eq:3.3}
	u(t)= e^{i t\mathcal D_{A,V}} u_0
	+ \int_0^t e^{i(t-\tau)\mathcal D_{A,V}} |u(\tau)|^{p-1}
	u(\tau)\,d\tau.
	\end{equation}
Here $e^{it\mathcal D_{A,V}}$ stands for the propagator
associated with linear hGLK equation,
namely \eqref{eq:1.1} with trivial RHS.
Briefly speaking, $e^{it\mathcal D_{A,V}}f$
solves the linear hGLK with $f$ as initial datum.
By Lemma \ref{Lemma:A.2},
$e^{i t\mathcal D_{A,V}}$ is a uniformly bounded operator on $H^s$ for $n=1$ and $s =1$
or for $n=2$, $3$ and $s=2$.
A standard fixed point argument implies that \eqref{eq:3.3}
has a solutions in $C([0,T);H^1)$ if $n=1$ and in $C([0,T);H^2)$ if $n=2,3$.
\end{proof}

\section{Commutator Estimates}
In this section,
we assume A1, A2, A3, H1, and H2.

\subsection{Preliminary}

The following representation is essential for our approach to study commutator estimates.
\begin{Lemma}
\label{Lemma:4.1}
	\begin{align*}
	&\big \langle g, [ (\mathcal H_{A,V})^{s/2}, f ] h \big \rangle\\
	&= - C_0(s) \int_0^{\infty} \lambda^{s/2}
	\langle (\lambda + \mathcal H_{A,V})^{-1} g,
	[\mathcal H_{A,V}, f] (\lambda + \mathcal H_{A,V})^{-1} h \rangle\, d\lambda.
	\end{align*}
\end{Lemma}

\begin{proof}
By \eqref{eq:3.1}, we have
	\begin{align*}
	&\big \langle g, [ \mathcal H_{A,V}^{s/2}, f ] h \big \rangle\\
	&= C_0(s) \int_0^\infty
	\lambda^{s/2-1} \langle
	g,[ \mathcal H_{A,V} (\lambda + \mathcal H_{A,V})^{-1}, f]h \rangle\, d\lambda\\
	&= C_0(s) \left( \int_0^\infty \lambda^{s/2}
	\langle g, [(\lambda + \mathcal H_{A,V})^{-1}, f ] h \rangle\,d\lambda \right).
	\end{align*}
Therefore Lemma \ref{Lemma:2.3} implies Lemma \ref{Lemma:4.1}.
\end{proof}

\begin{Lemma}
\label{Lemma:4.2}
Let $\mathcal A$ be a non-negative self-adjoint operator.
For $\sigma > \frac14,$
	\[
	\| (\cdot)^{\sigma - 3/4} \mathcal A^{1/4}
	(\cdot  + \mathcal A)^{-\sigma} f
	\|_{L^2((0,\infty); L^2)}
	\leq
	\bigg( \int_0^\infty \frac{\lambda^{2 \sigma - 3/2}}{(\lambda + 1)^{2\sigma}}
	\,d
	\lambda \bigg)^{1/2}
	\| f \|_{L^2}.
	\]
\end{Lemma}

\begin{proof}
Using the spectral measures $E_\mu$ for $\mathcal A$ (see \cite[Theorem VII.7]{RSI},
for instance),  we can write
\begin{align*}
&\| (\cdot)^{\sigma - 3/4} \mathcal A^{1/4} (\cdot + \mathcal A)^{-\sigma} f
\|_{L^2((0,\infty); L^2)}^2\\
&\leq \int_0^\infty
\int_0^\infty
\frac{\lambda^{2\sigma - 3/2} \mu^{1/2}}{(\lambda + \mu)^{2\sigma}}\,
d \| E_\mu(f)\|_{L^2}^2\, d\lambda\\
&= \int_0^\infty
\frac{\lambda^{2\sigma - 3/2}}{(\lambda + 1)^{2\sigma}}\, d\lambda
\thinspace \| f\|_{L^2}^2.
\end{align*}
\end{proof}

In the next lemma, we recall the well-known result
that the function $t\to t^s$, $s\in[0,1]$,
is operator monotone on the set of bounded operators in a Hilbert space.
One can see \cite{L34} for the original  matrix-valued version of the statement,
\cite[Proposition 4.2.8]{KR97} for the case $s=1/2$ and \cite{P72}
for a short proof of the general case.
See also \cite{H51,K52}.

\begin{Lemma}[\cite{L34}, \cite{P72}, \cite{KR97}]
\label{Lemma:4.3}
Let $(\mathcal A_1,D(\mathcal A_1))$ and $(\mathcal A_2,D(\mathcal A_2))$
be two positive self-adjoint operators on $L^2$
satisfying $D(\mathcal A_2) \subset D(\mathcal A_1)$ and
	\[
	\langle f,
	\mathcal A_1 f \rangle
	\leq \langle f, \mathcal A_2 f \rangle.
	\]
Then
	\begin{align}
	\langle f, \mathcal A_1^{s} f\rangle
	\leq \langle f, \mathcal A_2^{s} f\rangle
	\label{eq:4.1}
	\end{align}
for $0 < s \leq 1.$ Moreover, if $\mathcal A_1$ is invertible, so is $\mathcal A_2$,
and
	\[
	\langle f, \mathcal A_2^{-1} f \rangle \leq \langle f, \mathcal A_1^{-1} f \rangle.
	\]
\end{Lemma}


\medskip

\subsection{Fractional Leibniz Rules}
Here we collect some useful Leibniz rules
for fractional power of the classical Laplace operator.

\begin{Lemma}[{\cite[Proposition 4.1.A]{T91}}]\label{Lemma:4.4}
Let $f$ be a Lipschitz function.
Then for any $g \in H^1$
\[
\| [(-\Delta)^{1/2},f] g \|_{L^2}
\leq C \| f \|_{\mathrm{Lip}} \| g\|_{L^2}.
\]
\end{Lemma}

\begin{Lemma}[{\cite[Lemma A.12]{KPV93}}]
\label{Lemma:4.5}
For $0 < s < 1$, there exists $C > 0$ such that
\begin{align*}
&\| (-\Delta)^{s/2}(fg)
- f (-\Delta)^{s/2} g
- g (-\Delta)^{s/2} f
\|_{L^2} \leq C  \| f\|_{L^\infty}
\| (-\Delta)^{s/2} g\|_{L^2}.
\end{align*}
\end{Lemma}

\begin{Remark}
In \cite{L16}, one can find the refined estimate
\begin{align*}
&\| (-\Delta)^{s/2}(fg)
- f (-\Delta)^{s/2} g
- g (-\Delta)^{s/2} f
\|_{L^2} \lesssim  \| f\|_{\mathrm{BMO}}
\| (-\Delta)^{s/2} g\|_{L^2}.
\end{align*}
and more general estimates,
but for simplicity, we use only Lemma \ref{Lemma:4.5}.
\end{Remark}

In the sequel, we shall also need a generalization, obtained in \cite{GO14},
of the classical Kato-Ponce estimate,
introduced in the seminal and well celebrated work \cite{KP88}. We recall it.

\begin{Lemma}[{\cite[Theorem 1]{GO14}}]
\label{Lemma:4.6}
Let $1/2 < r < \infty$, $1 < p_1,p_2,q_1,q_2 \leq \infty$ satisfying
	\[
	\frac{1}{r}
	= \frac{1}{p_1} + \frac{1}{q_1}
	= \frac{1}{p_2} +
	\frac{1}{q_2}.
	\]
For $s > \max\{0,n/r-n\}$ or $s \in 2 \mathbb N$ (the set of positive even integers), there exists $C > 0$ such that
	\begin{align*}
	&\| (-\Delta)^{s/2}(fg)\|_{L^r}\\
	&\leq C \| (-\Delta)^{s/2} f\|_{L^{p_1}} \|g\|_{L^{q_1}}
	+ C \| f\|_{L^{p_2}} \| (-\Delta)^{s/2} g\|_{L^{q_2}}.
	\end{align*}
\end{Lemma}

\medskip

\subsection{Key estimate for  Proposition \ref{Proposition:2}}
The purpose of this subsection is to show that
the commutator between $\mathcal D_{A,V}$ and a localized weight function
is realized as a bounded operator in $L^2$
under the following assumptions:
	\begin{align}
	&\|(\lambda  + \mathcal H_{A,V} )^{-\sigma} f\|_{L^2}
	\lesssim \|(\lambda -\Delta)^{-\sigma} f\|_{L^2},
	\quad \forall f \in L^2,\, \lambda > 0
	\label{eq:4.2}
	\end{align}
for some $1/4 < \sigma \leq 1;$
	\begin{align}
	\| \mathcal D_{A,V} f \|_{L^2}
	&\lesssim \| (-\Delta)^{1/2} f \|_{L^2},
	\quad \forall f \in H^1;
	\label{eq:4.3}\\
	\| (-\Delta)^{1/4} f \|_{L^2}
	&\lesssim \|
	\mathcal H_{A,V}^{1/4} f \|_{L^2},
	\quad \forall f \in H^{1/2};
	\label{eq:4.4}
	\end{align}
for any $g,h \in H^{1/2}$
	\begin{equation}\label{eq:4.5}
	\begin{aligned}
	\langle g,
	\nabla (A (\nabla f) h) \rangle
	&\lesssim
	\|(-\Delta)^{1/4} \nabla f\|_{L^\infty}
	\| (-\Delta)^{1/4} g\|_{L^2} \| h \|_{L^2}\\
	&+\| \nabla f\|_{L^\infty}
	\| (-\Delta)^{1/4} g\|_{L^2} \| (-\Delta)^{1/4} h \|_{L^2}.
	\end{aligned}
	\end{equation}

\begin{Lemma}
\label{Lemma:4.7}
Assume {\rm A1}, {\rm A2}, {\rm A3}, {\rm H1}, and {\rm H2}.
Let $A$ and $V$ satisfy the properties
\eqref{eq:4.2}, \eqref{eq:4.3}, \eqref{eq:4.4}, and \eqref{eq:4.5}.
Then for any $ j \in \mathbb Z$,
	\begin{align}
	\|[\mathcal D_{A,V}, P_j f ] P_{\leq j} h \|_{L^2}
	&\lesssim 2^j \|P_j f\|_{L^\infty} \|h\|_{L^2},
	\label{eq:4.6}\\
	\left| \langle P_{>
	j} g , [\mathcal D_{A,V}, P_j f ] P_{ > j} h \rangle \right|
	&\lesssim 2^j \|P_j f\|_{L^\infty} \|g\|_{L^2} \|h\|_{L^2}.
	\label{eq:4.7}
	\end{align}
\end{Lemma}

\begin{proof}
We first prove \eqref{eq:4.6}.
The same relation given at the beginning of the proof of Lemma \ref{Lemma:2.2}
and the triangular inequality gives
	\begin{equation}\label{eq:4.8}
	\begin{aligned}
	&\|
	[\mathcal D_{A,V}, P_j f] P_{\leq j} h \|_{L^2}\\
	&\leq
	\| \mathcal D_{A,V} (-\Delta)^{-1/2} [ (-\Delta)^{1/2}, P_j f] P_{\leq j} h \|_{L^2}\\
	&+
	\| [\mathcal D_{A,V}(-\Delta)^{-1/2}, P_j f](-\Delta)^{1/2} P_{\leq j} h \|_{L^2}.
	\end{aligned}
	\end{equation}
By \eqref{eq:4.3} and Lemma \ref{Lemma:4.4},
the first term on the R.H.S. of \eqref{eq:4.8} is estimated as
	\begin{align*}
	&\| \mathcal D_{A,V} (-\Delta)^{-1/2}
	[ (-\Delta)^{1/2}, P_j f] P_{\leq j} h \|_{L^2}\\
	&\leq \|  [ (-\Delta)^{1/2}, P_j f] P_{\leq j} h \|_{L^2}
	\lesssim 2^j \| P_j f \|_{L^\infty} \| P_{\leq j} h \|_{L^2},
	\end{align*}
where we have used the fact that
	\[
	\| \nabla P_j f \|_{L^\infty}
	\lesssim 2^{j} \| P_j f \|_{L^\infty}.
	\]
By \eqref{eq:4.3}, the second term on the R.H.S. of \eqref{eq:4.8} is estimated as
	\begin{equation*}
	\begin{aligned}
	&\| [\mathcal D_{A,V} (-\Delta)^{-1/2}, P_j f](-\Delta)^{1/2} P_{\leq j} h \|_{L^2}\\
	&\lesssim \| P_j f \|_{L^\infty} \|(-\Delta)^{1/2} P_{\leq j} h
	\|_{L^2}\lesssim 2^{j+1} \| P_j f \|_{L^\infty} \| h \|_{L^2}.
	\end{aligned}
	\end{equation*}

\noindent We next prove \eqref{eq:4.7}.
By Lemma \ref{Lemma:4.1}, it is enough to show
	\begin{align}\label{eq:4.9}
	&\left|
	\int_0^{\infty} \lambda^{1/2}
	\big \langle (\lambda + \mathcal H_{A,V})^{-1} P_{> j} g,
	\nabla \cdot A (\nabla P_j f)(\lambda + \mathcal H_{A,V})^{-1}
	P_{> j}h \big \rangle\,
	d\lambda \right|
	\nonumber\\
	&\lesssim 2^j \|P_j f\|_{L^\infty} \|g\|_{L^2} \|h\|_{L^2}.
	\end{align}
By \eqref{eq:4.4} and \eqref{eq:4.5},
the L.H.S. of \eqref{eq:4.9} is estimated by
	\begin{align*}
	&2^j \|P_j f\|_{L^\infty} \|g\|_{L^2} \|h\|_{L^2}\\
	&\lesssim 2^{3j/2} \| A \|_{L^\infty} \|P_j f\|_{L^\infty}\\
	&\times \int_0^{\infty} \lambda^{1/2}
	\|\mathcal H_{A,V}^{1/4} (\lambda + \mathcal H_{A,V})^{-1} P_{> j} g \|_{L^2}
	\|(\lambda + \mathcal H_{A,V})^{-1} P_{> j} h \|_{L^2} \,d\lambda\\
	&+ 2^{j} \| A \|_{W^{1,\infty}} \|P_j f\|_{L^\infty}\\
	& \times \int_0^{\infty} \lambda^{1/2}
	\|\mathcal H_{A,V}^{1/4}(\lambda + \mathcal H_{A,V})^{-1} P_{> j} g \|_{L^2}
	\|\mathcal H_{A,V}^{1/4} (\lambda + \mathcal H_{A,V})^{-1} P_{> j} h \|_{L^2}
	\,d\lambda.
	\end{align*}
Then,
by Lemma \ref{Lemma:4.2},
the first integral on the R.H.S. of the last inequality is estimated by
	\begin{align*}
	&2^{3j/2}\int_0^{\infty} \lambda^{1/2}
	\|\mathcal H_{A,V}^{1/4}
	(\lambda + \mathcal H_{A,V})^{-1} P_{> j} g \|_{L^2}
	\|(\lambda + \mathcal H_{A,V})^{-1} P_{> j} h \|_{L^2} \,d\lambda\\
	&\lesssim 2^{3j/2}
	 \int_0^{\infty}
	\lambda^{\sigma - 1/2}
	\|\mathcal H_{A,V}^{1/4} (\lambda + \mathcal H_{A,V})^{-1} P_{> j} g \|_{L^2}
	\|(\lambda - \Delta )^{-\sigma} P_{> j} h \|_{L^2}
	\,d\lambda\\
	&\lesssim 2^{j}
	\| (\cdot)^{1/4} \mathcal H_{A,V}^{1/4} (\cdot + \mathcal H_{A,V})^{-1} P_{> j} g
	\|_{L^2(0,\infty;L^2)}\\
	&\times
	\|(\cdot)^{\sigma - 3/4} (-\Delta)^{1/4} (\cdot -\Delta)^{-\sigma} P_{> j} h
	\|_{L^2(0,\infty;L^2)}\\
	&\lesssim 2^{j} \| g\|_{L^2} \| h\|_{L^2}
	\end{align*}
with $1/4 < \sigma \leq 1$ satisfying \eqref{eq:4.2}.
The second integral is also estimated by
	\begin{align*}
	&
	2^{j} \int_0^{\infty} \lambda^{1/2}
	\|\mathcal H_{A,V}^{1/4}(\lambda + \mathcal H_{A,V})^{-1} P_{> j} g \|_{L^2}
	\|\mathcal H_{A,V}^{1/4} (\lambda + \mathcal H_{A,V})^{-1} P_{> j} h \|_{L^2}
	\,d\lambda\\
	&\lesssim
	2^{j}
	\|
	(\cdot)^{1/4} \mathcal H_{A,V}^{1/4} (\cdot + \mathcal H_{A,V})^{-1} P_{> j} g
	\|_{L^2(0, \infty;L^2)}\\
	&\times \| (\cdot)^{1/4} \mathcal H_{A,V}^{1/4}
	(\cdot + \mathcal H_{A,V})^{-1} P_{> j} h \|_{L^2(0,\infty;L^2)}\\
	&\lesssim 2^{j} \| g\|_{L^2} \| h\|_{L^2}.
	\end{align*}
\end{proof}


\subsection{Proof of Proposition \ref{Proposition:2}}
We are now in a position to prove Proposition \ref{Proposition:2}.
We treat separately the two cases.
\begin{proof}[Proof of Case 1.]
At first, we show that Lemma \ref{Lemma:4.7} implies \eqref{eq:1.7} with $V\equiv0$.
\eqref{eq:4.2}, \eqref{eq:4.3}, and \eqref{eq:4.4} follow from A1.
Indeed, \eqref{eq:1.5} implies
	\[
	C_1 \| \nabla f \|_{L^2}^2
	\leq \langle f, \mathcal H_{A,0} f \rangle
	= \| \mathcal D_{A,0} f \|_{L^2}^2
	\leq C_2 \| \nabla f
	\|_{L^2}^2
	\]
which coincides with \eqref{eq:4.3}.
Therefore, Lemma \ref{Lemma:4.3} can be applied with $\mathcal H_{A,0}$ and $-\Delta$.
Hence, the relation \eqref{eq:4.1}, with $s= 1/2$,
$\mathcal A_1 = \mathcal H_{A,0}$ and $\mathcal A_2 = - C \Delta$,
coincides with \eqref{eq:4.4}.
Moreover \eqref{eq:1.5} implies that one can find two constants $c, C$ with  $0 < c \leq 1 \leq C$ such that
	\[
	c\,\langle f , (\lambda -\Delta) f\rangle
	\leq \langle f , (\lambda  + \mathcal H_{A,0} ) f\rangle
	\leq C \langle f , (\lambda - \Delta ) f\rangle
	\]
for any $f \in H^2$ and $\lambda \geq 0$.
Then, Lemma \ref{Lemma:4.3} implies that for any $f \in
L^2$
	\[
	\langle f, (\lambda  + \mathcal H_{A,0} )^{-1} f\rangle
	\leq \langle f, c^{-1} (\lambda -\Delta)^{-1} f\rangle,
	\]
which coincides with \eqref{eq:4.2} with $\sigma = 1/2$.

\eqref{eq:4.5} may be obtained by decomposing
$\partial_j (A_{j,k} (\partial_k f) h )$ as follows:
	\begin{align}\label{eq:4.10}
	\partial_j (A_{j,k} (\partial_k f) h)
	& = (-\Delta)^{1/4} R_j (-\Delta)^{1/4} (A_{j,k} (\partial_k f) h)
	\nonumber\\
	&= (-\Delta)^{1/4} R_j
	A_{j,k} (-\Delta)^{1/4} ((\partial_k f) h)
	\nonumber\\
	&+ (-\Delta)^{1/4} R_j ((-\Delta)^{1/4} A_{j,k})(\partial_k f) h
	\nonumber\\
	&+ (-\Delta)^{1/4} R_j B(A_{j,k}, (\partial_k f)h),
	\end{align}
where
	\[
	\mathfrak F (R_j f) = \frac{\xi_j}{|\xi|} \hat f(\xi)
	\]
is, up to a complex constant, the standard Riesz transform, and
	\[
	B(A_{j,k},\partial_k f)
	:= (-\Delta)^{1/4} (A_{j,k}
	\partial_k f)
	- A_{j,k} (-\Delta)^{1/4} \partial_k f
	- \partial_k f (-\Delta)^{1/4} A_{j,k}.
	\]
The first term on the R.H.S. of \eqref{eq:4.10} is easily estimated
by the H\"older inequality and Lemma \ref{Lemma:4.6}.
Here we recall that \eqref{eq:1.5} implies $\|A_{j,k}\|_{L^\infty} < \infty$.
The other terms are estimated similarly,
since by Lemma \ref{Lemma:4.5} and \eqref{eq:1.6},
we have
	\begin{equation*}
	\| B(A_{j,k}, (\partial_k f)h) \|_{L^2}
	\lesssim \| A \|_{L^\infty}
	\| (-\Delta)^{1/4} ((\partial_k f) h) \|_{L^2}
	\end{equation*}
and
	\begin{equation*}
	\| ((-\Delta)^{1/4} A_{j,k})(\partial_k f) h \|_{L^2}
	\lesssim \|
	(-\Delta)^{1/4} ((\partial_k f) h) \|_{L^2},
	\end{equation*}
respectively. 

We now show Proposition \ref{Proposition:2} with $V\equiv0$.
Since
	\begin{align*}
	&\langle g, [\mathcal D_{A,0},f] h \rangle\\
	&= \langle P_{\leq j}g,
	[\mathcal D_{A,0}, f] h \rangle
	+ \langle g, [\mathcal D_{A,0}, f] P_{\leq j}h \rangle
	+ \langle P_{> j} g, [\mathcal D_{A,0},f]
	P_{> j}h \rangle\\
	&= - \overline{\langle h , [\mathcal D_{A,0}, f] P_{\leq j}g \rangle}
	+ \langle g, [\mathcal D_{A,0}, f] P_{\leq j}h \rangle
	+ \langle P_{> j} g, [\mathcal D_{A,0},f] P_{> j}h \rangle,
	\end{align*}
Lemma \ref{Lemma:4.2} implies the estimate.

We next show \eqref{eq:1.7} with $V \nequiv 0$.
\eqref{eq:1.7} follows from the fact that
for any $g \in H^1$,
	\[
	\| (\mathcal D_{A,0} - \mathcal D_{A,V}) g \|_{L^2}
	\leq \| V \|_{L^{q,\infty}} \| g \|_{L^2}.
	\]
Indeed, by
Lemma \ref{Lemma:4.1},
	\begin{align*}
	&C_0(1/2)^{-1}(\mathcal D_{A,0} - \mathcal D_{A,V}) g\\
	&=  \left(\int_0^\infty \lambda^{1/2}
	((\lambda + \mathcal H_{A,0})^{-1} - (\lambda + \mathcal H_{A,V})^{-1})\,
	d \lambda\right) g\\
	&=\left( \int_0^1 \lambda^{1/2}
	((\lambda + \mathcal H_{A,0})^{-1}
	- (\lambda + \mathcal H_{A,V})^{-1}) \,d \lambda\right) g\\
	&+\left(\int_1^\infty \lambda^{1/2}
	(\lambda + \mathcal H_{A,V})^{-1} V
	(\lambda + \mathcal H_{A,0})^{-1} \,d \lambda\right) g.
	\end{align*}
The $L^2$-norm of the first integral on the R.H.S. of the last equality
is shown to be bounded
by the fact that for any non-negative self-adjoint operator $\mathcal A$
	\[
	\| (\lambda + \mathcal A)^{-1} g\|_{L^2}
	\leq \lambda^{-1} \| g\|_{L^2}.
	\]
By \eqref{eq:1.5} and Lemma \ref{Lemma:4.3},
	\begin{align*}
	\| V (\lambda + \mathcal H_{A,0})^{-1} g \|_{L^2}
	&\lesssim \|
	V\|_{L^{q,\infty}}
	\|(-\Delta)^{n/2q}(\lambda + \mathcal H_{A,0})^{-1} g \|_{L^2}\\
	&\lesssim \| V\|_{L^{q,\infty}}
	\|\mathcal H_{A,0}^{n/2q}(\lambda + \mathcal H_{A,0})^{-1} g \|_{L^2}\\
	&\lesssim \| V\|_{L^{q,\infty}} \lambda^{-1+n/2q} \| g \|_{L^2}.
	\end{align*}
Then, the $L^2$-norm of the second integral is shown to be bounded by
	\[
	\int_1^\infty \lambda^{-3/2 + n/2q} \,d \lambda<\infty.
	\]
\end{proof}

\begin{proof}[Proof of Case 2.]
\eqref{eq:1.8} follows if we are able to show that
	\begin{align}
	-\Delta \sim - \Delta_{A,V},
	\label{eq:4.11}
	\end{align}
where the equivalence is in the sense of bilinear forms.
Indeed, if \eqref{eq:4.11} is shown, then \eqref{eq:4.2},
\eqref{eq:4.3}, and \eqref{eq:4.4} are satisfied
and therefore Lemma \ref{Lemma:4.7} implies \eqref{eq:1.8}.
The relation \eqref{eq:4.11} is proved as follows:
	\begin{align*}
	\langle f, -\Delta_{A,V} f\rangle
	&\geq (1 - \theta) \langle A \nabla f, \nabla f \rangle\\ 
	&\geq C_1 ( 1 - \theta)  \| \nabla f \|_{L^2}^2,\\
	\langle f, -\Delta_{A,V} f\rangle
	&\leq C_2 \| \nabla f \|_{L^2}^2 + \| |V|^{1/2} f\|_{L^2}^2\\
	&\leq C_2 \| \nabla f \|_{L^2}^2
	+ C^2 \| |V|^{1/2}\|_{L^{n,\infty}}^2 \| (-\Delta)^{1/2}f\|_{L^2}^2.
	\end{align*}
\end{proof}


\section{The finite time blow-up result}


Theorem \ref{Theorem:4} may be  concluded  be means of the following ODE argument.

\begin{Lemma}
\label{Lemma:5.1}
Let $A, B > 0$ and $q>1$.
If $f \in C^1([0,T); \mathbb R^+)$ satisfies $f(0) > 0$ and
	\[
	f' + A f
	= B f^q \quad \mbox{on $[0,T)$ for some $T>0$},
	\]
then
	\[
	f(t)
	= e^{-A t}
	\bigg( f(0)^{-(q-1)} + A^{-1}
	B e^{-A(q-1)t}
	- A^{-1} B \bigg)^{-\frac{1}{q-1}}.
	\]
Moreover, if $f(0) > A^{\frac 1 {q-1}} B^{- \frac 1 {q-1}}$,
then $T < - \frac 1 {A (q-1)} \log ( 1 - A B^{-1} f(0)^{-q+1} )$.
\end{Lemma}

\begin{proof}
For completeness, we sketch the proof.
Let $f = e^{-A t} g$.
Then
	\[
	g' = B e^{-A(q-1)t} g^q
	\]
and therefore,
	\[
	\frac{1}{1-q} \bigg( g^{1-q}(t) - g^{1-q}(0) \bigg)
	= \frac{B}{A(1-q)} ( e^{-A(q-1)t} - 1).
	\]
The conclusion follows straightforward.
\end{proof}


We exploit Lemma \ref{Lemma:5.1} in the proof of  Theorem \ref{Theorem:4}.
\begin{proof}[Proof of Theorem  \ref{Theorem:4}.]
\noindent \emph{Case 1.}
Let $w \in B_{\infty,1}^1(\mathbb R)$ be a non-negative function
satisfying $1/w \in L^{\infty} \cap L^2$.
We put $u = vw$.
Then $v$ satisfies
	\begin{align}
	\partial_t v + \frac{i}{w} [\mathcal D_{A,V}, w] v
	= w^{p-1} |v|^{p-1} v.
	\label{eq:5.1}
	\end{align}
By multiplying $\overline v$ on the both hand sides of \eqref{eq:5.1},
integrating the resulting equation, and taking the real part,
	\begin{align}
	&\frac{1}{2} \frac{d}{dt} \| v(t)\|^2_{L^2}
	\nonumber \\
	&\geq \| w^{\frac{p-1}{p+1}} v(t) \|_{L^{p+1}}^{p+1}
	- \| 1/w \|_{L^\infty} \| [\mathcal D_{A,V}, w] v \|_{L^2}
	\| v\|_{L^2}
	\nonumber \\
	&\geq \| 1/w \|_{L^{2}}^{-p+1} \| v\|_{L^2}^{p+1}
	- \| 1/w \|_{L^\infty} \| [\mathcal D_{A,V}, w] v \|_{L^2}
	\| v\|_{L^2}
	\nonumber \\
	&\geq \| 1/w \|_{L^{2}}^{-p+1} \| v\|_{L^2}^{p+1}
	- C\| 1/w\|_{L^\infty}\|w\|_{B^1_{\infty,1}}\| v\|_{L^2}^2,
	\label{eq:5.2}
	\end{align}
where we have used that
	\[
	\| v \|_{L^2}
	\leq
	\| 1/w^{\frac{p-1}{p+1}} \|_{L^{\frac{2(p+1)}{p-1}}}
	\| w^{\frac{p-1}{p+1}} v(t) \|_{L^{p+1}}
	\leq
	\| 1/w \|_{L^{2}}^{\frac{p-1}{p+1}}
	\|
	w^{\frac{p-1}{p+1}} v(t) \|_{L^{p+1}}.
	\]
By \eqref{eq:5.2}, we apply Lemma \ref{Lemma:5.1} with
	\begin{align*}
	A &= C \| 1/w \|_{L^\infty} \| w \|_{B_{\infty,1}^1},\\
	B &= \| 1/w \|_{L^{2}}^{-p+1}.
	\end{align*}
Then \eqref{eq:1.9} implies that $\|v(t)\|_{L^2}$ is not uniformly controlled.

\noindent\emph{Case 2.}
We rescale $w \in \dot B_{\infty,1}^1$ as $w_R = w(\cdot / R)$ with $R > 0$.
Then by \eqref{eq:5.2},
	\begin{align*}
	&\frac{1}{2} \frac{d}{dt} \| v(t)\|_{L^2}^2\\
	&\geq \| 1/w_R
	\|_{L^{2}}^{-p+1} \| v\|_{L^2}^{p+1}
	- \| 1 /w_R \|_{L^\infty} \| [(-\Delta_{A,V})^{1/2}, w_R] v \|_{L^2}
	\| v\|_{L^2}\\
	&\geq R^{-(p-1)/2} \| 1/w
	\|_{L^{2}}^{-p+1} \| v\|_{L^2}^{p+1}
	- C R^{-1} \| 1/w \|_{L^\infty} \| w \|_{\dot B_{\infty,1}^1} \| v\|_{L^2}^2.
	\end{align*}
We apply Lemma
\ref{Lemma:5.1} with
	\begin{align*}
	A &= C R^{-1} \| 1/w \|_{L^\infty} \| w \|_{B_{\infty,1}^1},\\
	B &= R^{-(p-1)/2} \| 1/w \|_{L^{2}}^{-p+1},
	\end{align*}
which means $AB^{-1} \sim R^{-1+(p-1)/2}$.
Therefore, if $1<p < 3$, $A B^{-1} \to 0$ as
$R \to \infty$ and this shows Theorem \ref{Theorem:4}.
\end{proof}

\appendix
\section{Equivalence of Sobolev norms}
\label{section:A}
We show the equivalence of the standard $H^s$-norms
(for $s=1,2$) and the ones induced by the Hamiltonian $\mathcal H_{A,V}$.
We begin with simple a priori estimates that imply the equivalence of $H^1$ norms.
\begin{Lemma}
\label{Lemma:A.1}
Assume $\mathrm{H2}$.
If $V \in L^{q,\infty}(\mathbb{R}^n)+ L^{\infty}(\mathbb{R}^n)$
with $q > \max\{1, n/2\}$,
then for any  $\alpha \in (0,1)$ there exists $C>0$, so that
for any $f \in C_c^{\infty}(\mathbb R^n)$,
	\begin{equation}\label{eq:A.1}
	\langle ( - \alpha  \nabla \cdot A \nabla -  |V| ) f, f \rangle
	\geq -C \|f\|_{L^2}^2
	\end{equation}
and
	\begin{equation}\label{eq:A.2}
	\langle A \nabla f, \nabla f \rangle
	+ \langle  V f, f \rangle
	+ \|f\|_{L^2}^2
	\sim \|f\|_{H^1}^2.
	\end{equation}
\end{Lemma}

\begin{proof}
We know that uniform ellipticity assumption implies
$$ \langle \left( -   \nabla \cdot A \nabla\right) f, f \rangle \sim \| f \|_{\dot{H}^1}.$$
We need to prove the inequality
	\begin{equation}\label{eq:A.3}
	\int_{\mathbb{R}^n} |V| |f|^2dx  \lesssim \|f \|^2_{\dot{H}^s}
	\end{equation}
with $0 < s < 1$,
since this estimate and  the Gagliardo-Nirenberg interpolation inequality
	\[
	\| f \|_{\dot{H}^s} \lesssim \| f \|^s_{\dot{H}^1} \| f \|^{1-s}_{L^2}
	\]
imply
	\[
	\int_{\mathbb{R}^n} |V| |f|^2 dx
	\leq \langle ( - \alpha  \nabla \cdot A \nabla ) f, f \rangle + C \|f\|_{L^2}^2,
	 \]
so we have \eqref{eq:A.1} and \eqref{eq:A.2}.

In order to prove \eqref{eq:A.3}, we take
	\[
	\frac{1}{r} = \frac{1}{2} - \frac{1}{2q}, \ s= \frac{n}{2q}
	\]
and then we can write
\[
\left( \int_{\mathbb{R}^n} |V| |f|^2dx \right)^{1/2}
\lesssim \| |V|^{1/2}\|_{L^{2q,\infty}} \|f\|_{L^{r,2}}
\lesssim \| V \|_{L^{q,\infty}}^{1/2} \|f \|^2_{\dot{H}^s}
\]
due to H\"older inequality in Lorentz spaces and Sobolev embedding. The requirement $0 < s < 1$ is fulfilled due to the assumption $q>n/2.$
\end{proof}

\begin{Lemma}
\label{Lemma:A.2}
Assume $\mathrm{A1}$, $\mathrm{A2}$, $\mathrm{H1}$, and $\mathrm{H2}$.
Then one can find positive constants $C_1 < C_2$ so that for any $f \in C_c^\infty(\mathbb R^n)$,
	\begin{align}
	C_1 \| f \|_{H^2}
	&\leq \| \mathcal - \Delta_{A,V} f \|_{L^2} + \| f \|_{L^2}
	\leq C_2 \| f \|_{H^2}.
	\label{eq:A.4}
	\end{align}
\end{Lemma}

\begin{proof}
The right inequality of \eqref{eq:A.4} follows directly from
the representation of $\Delta_{A,V}$.
Indeed
	\begin{align*}
	\Delta_{A,V} f
	&= ( \nabla A ) \cdot \nabla f + \sum_{j,k=1}^n A_{j,k}(x) \partial_j
	\partial_k f
	+ V f.
	\end{align*}
Further we can take
$$ \frac{1}{r} = \frac{1}{2} - \frac{1}{q}, \ s= \frac{n}{q}$$
and then we can write
	\begin{align}
	\|Vf\|_{L^2}
	\lesssim \|V\|_{L^{q,\infty}} \|f\|_{L^{r,2}}
	\lesssim  \|V\|_{L^{q,\infty}}\|f \|_{H^s}
	\label{eq:A.5}
	\end{align}
with $ s \in (0,2)$
so interpolation yields the right-side estimate.

Next we show the left inequality of \eqref{eq:A.4} with $V=0$.
By $\mathrm{A1}$,
	\begin{align*}
	&C_1 \| ( -\Delta) f \|_{L^2}^2\\
	&= C_1
	\int_{\mathbb R^n} \overline{(-\Delta) f(x)} (-\Delta) f(x)\, dx\\
	&\leq
	\int_{\mathbb R^n} \overline {(-\Delta)^{1/2} f(x)}
	(- \Delta_{A,0})
	(-\Delta)^{1/2}f(x)\, dx\\
	&= \int_{\mathbb R^n} \overline {(-\Delta) f(x)} (- \Delta_{A,0}) f(x)\, dx
	+ \int_{\mathbb R^n} \nabla \overline
	{(-\Delta)^{1/2} f(x)}
	\cdot \mathcal G_A f(x)\, dx\\
	&\leq \| ( -\Delta) f \|_{L^2}
	( \| - \Delta_{A,0} f \|_{L^2} + \| \mathcal G_{A} f \|_{L^2}),
	\end{align*}
where $ \mathcal G_A = [(-\Delta)^{1/2}, A] \nabla$ so that
	\[
	\nabla \cdot \mathcal G_A f
	= [( - \Delta_{A,0}), (-\Delta)^{1/2}] f
	= \nabla [A, (-\Delta)^{1/2}] \nabla f.
	\]
Then, by Lemma \ref{Lemma:4.4}, the Gagliardo-Nirenberg and the Young inequalities,
	\[
	\|\mathcal G_A f\|_{L^2}
	\leq C \| \nabla A \|_{L^\infty} \| \nabla f \|_{L^2}
	\leq \frac{1}{2} \| (-\Delta) f \|_{L^2} + C \|f\|_{L^2},
	\]
which in turn implies
	\[
	\| (-\Delta) f \|_{L^2}
	\leq C \| f \|_{L^2}
	+ C \| - \Delta_{A,0} f \|_{L^2}.
	\]
This inequality and \eqref{eq:A.5} prove the left estimate in \eqref{eq:A.4}.
\end{proof}

\section{Estimate of the weight function}
Our choice of $w$ for the proof of the blow-up result
is $w(x) = \langle x \rangle^a$ with $a \in (1/2,1)$.
The lower bound of $a$ is required to guarantee that $1/w \in L^2(\mathbb R)$
for Theorem \ref{Theorem:4}.
The upper bound of $a$ follows from the following Proposition:

\begin{Proposition}
\label{Proposition:7}
For $a < 1$,
	\[
	\langle \cdot \rangle^a \in \dot B_{\infty,1}^1.
	\]
\end{Proposition}

\begin{proof}
We recall that $2^{-sj} P_j (-\Delta)^{s/2}$ is a bounded operator on $L^\infty$.
Therefore for $j \geq 0$,
	\[
	\|P_j (-\Delta)^{1/2} \langle x \rangle^a \|_{L^\infty}
	\lesssim 2^{-j} \| 2^j P_j (-\Delta)^{-1/2} \Delta \langle x \rangle^a \|_{L^\infty}
	\lesssim 2^{-j} \|\Delta \langle x \rangle^a \|_{L^\infty}
	\]
which implies $P_{\geq 0} \langle \cdot \rangle^a \in \dot B_{\infty,1}^1$.
Moreover, for $a > 0$ since
	\[
	\| P_j f \|_{L^\infty}
	\lesssim
	2^{jn/p} \| f \|_{L^p}
	\]
and
	\[
	|\nabla \langle x \rangle^a |
	\lesssim \langle x \rangle^{a-1},
	\]
by taking $ p = \frac{2n}{1-a}$
	\begin{align*}
	\|P_j (-\Delta)^{1/2} \langle x \rangle^a \|_{L^\infty}
	&\lesssim
	2^{jn/p} \| \nabla (-\Delta)^{-1/2} \nabla \langle x \rangle^a \|_{L^p}\\
	&\lesssim 2^{j(1-a)/2} \| \langle x \rangle^{-1} \|_{L^{2n}}^{1-a}.
	\end{align*}
Therefore
	\begin{align}
	P_{\leq 0} \langle x \rangle^a  \in \dot B_{\infty,1}^1.
	\label{eq:B.1}
	\end{align}
For $a \leq 0$, it is easy to see \eqref{eq:B.1}.
\end{proof}

\begin{Remark}
It is worth mentioning that the estimate above is valid in arbitrary dimension,
but we can use only $n=1$ in order to prove Theorem \ref{Theorem:4}.
\end{Remark}

\begin{Remark}
The upper bound for the function $a$ in Proposition \ref{Proposition:7} is optimal.
Indeed,
	\[
	\langle \cdot \rangle \not\in \dot B_{\infty,1}^1(\mathbb R^n)
	\]
for any $n$.
In order to show this,
we estimate the following equivalent norm for $\dot B_{\infty,1}^1(\mathbb R^n)$:
	\[
	|\mkern-1.5mu|\mkern-1.5mu|f|\mkern-1.5mu|\mkern-1.5mu|%
	_{\dot B_{\infty,1}^1(\mathbb R^n)}
	= \int_0^\infty
	\sup_{|y| < t} \| f(\cdot+y) - 2 f(\cdot) + f(\cdot-y) \|_{L^\infty(\mathbb R^n)}
	\frac{dt}{t^2}.
	\]
For details, see \cite[6.3.1. Theorem]{BL76}.
Then, by substituting $x = 0$, for $t \geq 1$,
	\begin{align*}
	&\sup_{|y| < t} \, \sup_{x \in \mathbb R^n}
	|\langle x+y\rangle - 2 \langle x \rangle + \langle x-y\rangle|\\
	&\geq \sup_{|y| < t}
	2 (\langle y\rangle - 1)\\
	&\geq 2 \bigg(\bigg(1+\frac{t^2}{4} \bigg)^{1/2} - 1\bigg)
	> \frac{t}{8},
	\end{align*}
where we have used the fact that
	\[
	1+ \frac{t^2}{4}
	\geq 1+ \frac{t^2}{8} + \frac{t^2}{256}
	\geq \bigg( 1+ \frac{t}{16}\bigg)^2.
	\]
Therefore,
	\begin{align*}
	|\mkern-1.5mu|\mkern-1.5mu| \langle \cdot \rangle |\mkern-1.5mu|\mkern-1.5mu|%
	_{\dot B_{\infty,1}^1(\mathbb R^n)}
	&\geq \int_{1}^\infty
	\sup_{|y| < t}
	\| \langle \cdot+y \rangle - 2 \langle \cdot \rangle + \langle \cdot-y \rangle \|%
	_{L^\infty(\mathbb R^n)}
	\frac{dt}{t^2}\\
	&\geq \frac{1}{8} \int_{1}^\infty \frac{dt}{t} = \infty.
	\end{align*}
\end{Remark}

\providecommand{\bysame}{\leavevmode\hbox to3em{\hrulefill}\thinspace}
\providecommand{\MR}{\relax\ifhmode\unskip\space\fi MR }
\providecommand{\MRhref}[2]{%
  \href{http://www.ams.org/mathscinet-getitem?mr=#1}{#2}
}
\providecommand{\href}[2]{#2}


\begin{thebibliography}{10}

\bibitem{ACGM16}
P.~Acquistapace, A.~P. Candeloro, V.~Georgiev, and M.~L. Manca,
  \emph{Mathematical phase model of neural populations interaction in
  modulation of {REM}/{NREM} sleep}, Math. Model. Anal. \textbf{21} (2016),
  no.~6, 794--810.

\bibitem{BL76}
J.~Bergh and J.~L\"ofstr\"om, \emph{Interpolation {S}paces. {A}n
  {I}ntroduction}, Springer-Verlag, Berlin-New York, 1976, Grundlehren der
  Mathematischen Wissenschaften, No. 223.

\bibitem{BMS02}
M.~Braverman, O.~Milatovich, and M.~Shubin, \emph{Essential selfadjointness of
  {S}chr\"odinger-type operators on manifolds}, Uspekhi Mat. Nauk \textbf{57}
  (2002), no.~4(346), 3--58.

\bibitem{D90}
E.~B. Davies, \emph{Heat {K}ernels and {S}pectral {T}heory}, Cambridge Tracts
  in Mathematics, vol.~92, Cambridge University Press, Cambridge, 1990.

\bibitem{FGO17c}
K.~Fujiwara, V.~Georgiev, and T.~Ozawa, \emph{Blow-up for self-interacting
  fractional {G}inzburg-{L}andau equation}, to appear in Dyn. Partial Differ.
  Equ.

\bibitem{FGO17b}
\bysame, \emph{On global well-posedness for nonlinear semirelativistic
  equations in some scaling subcritical and critical cases}, preprint.

\bibitem{GMS17}
V.~Georgiev, A.~Michelangeli, and R.~Scandone, \emph{On fractional powers of
  singular perturbations of the {L}aplacian}, to appear in J. Funct. Anal.

\bibitem{GO14}
L.~Grafakos and S.~Oh, \emph{The {K}ato-{P}once {I}nequality}, Comm. Partial
  Differential Equations \textbf{39} (2014), no.~6, 1128--1157.

\bibitem{H51}
E.~Heinz, \emph{Beitr\"age zur {S}t\"orungstheorie der {S}pektralzerlegung},
  Math. Ann. \textbf{123} (1951), 415--438.

\bibitem{KR97}
R.~V. Kadison and J.~R. Ringrose, \emph{Fundamentals of the {T}heory of
  {O}perator {A}lgebras. {V}ol. {I}}, Graduate Studies in Mathematics, vol.~15,
  American Mathematical Society, Providence, RI, 1997, Elementary theory,
  Reprint of the 1983 original.

\bibitem{KPST05}
T.~Kappeler, P.~Perry, M.~Shubin, and P.~Topalov, \emph{The {M}iura map on the
  line}, Int. Math. Res. Not. (2005), no.~50, 3091--3133.

\bibitem{K52}
T.~Kato, \emph{Notes on some inequalities for linear operators}, Math. Ann.
  \textbf{125} (1952), 208--212.

\bibitem{KP88}
T.~Kato and G.~Ponce, \emph{Commutator estimates and the {E}uler and
  {N}avier-{S}tokes equations}, Comm. Pure Appl. Math. \textbf{41} (1988),
  no.~7, 891--907.

\bibitem{KPV93}
C.~Kenig, G.~Ponce, and L.~Vega, \emph{The {C}auchy problem for the
  {K}orteweg-de {V}ries equation in {S}obolev spaces of negative indices}, Duke
  Math. J. \textbf{71} (1993), no.~1, 1--21.

\bibitem{Kur00}
Y.~Kuramoto, \emph{Chemical {O}scillations, {W}aves, and {T}urbulence},
  Springer Series in Synergetics, vol.~19, Springer-Verlag, Berlin, 1984.

\bibitem{La00}
N.~Laskin, \emph{Fractional quantum mechanics and {L}\'evy path integrals},
  Phys. Lett. A \textbf{268} (2000), no.~4-6, 298--305.

\bibitem{L16}
D.~Li, \emph{On {K}ato-{P}once and fractional {L}eibniz}, to appear in Rev.
  Mat. Iberoamericana (2016).

\bibitem{L34}
K.~L\"owner, \emph{{\"U}ber monotone {M}atrixfunktionen. ({G}erman)}, Trans.
  Amer. Math. Soc. \textbf{38} (1934), no.~1, 177--216.

\bibitem{P72}
G.~K. Pedersen, \emph{Some operator monotone functions}, Proc. Amer. Math. Soc.
  \textbf{36} (1972), 309--310.

\bibitem{RSI}
M.~Reed and B.~Simon, \emph{Methods of {M}odern {M}athematical {P}hysics. {I}.
  {F}unctional {A}nalysis}, Academic Press, New York-London, 1972.

\bibitem{RSII}
\bysame, \emph{Methods of {M}odern {M}athematical {P}hysics. {II}. {F}ourier
  {A}nalysis, {S}elf-adjointness}, Academic Press [Harcourt Brace Jovanovich,
  Publishers], New York-London, 1975.

\bibitem{TZ16}
V.~Tarasov and G.~Zaslavsky, \emph{Fractional dynamics of coupled oscillators
  with long-range interaction}, Chaos \textbf{16} (2016), 023110.

\bibitem{T91}
M.~E. Taylor, \emph{Pseudodifferential {O}perators and {N}onlinear {PDE}},
  Progress in Mathematics, vol. 100, Birkh\"auser Boston, Inc., Boston, MA,
  1991.

\bibitem{vS92}
W.~van Saarloos and P.~C. Hohenberg, \emph{Fronts, pulses, sources and sinks in
  generalized complex {G}inzburg-{L}andau equations}, Physica D \textbf{56}
  (1992), 303--367.

\end{thebibliography}
\end{document}